\documentclass[11pt,reqno]{amsart} 

\usepackage{amssymb,amsmath,latexsym}
\usepackage{cite} 
\usepackage{graphicx}
\usepackage{epstopdf}
\usepackage{float}
\usepackage[height=190mm,width=130mm]{geometry} 
\usepackage{subfigure}
\theoremstyle{plain}
\newtheorem{theorem}{Theorem}

\theoremstyle{definition}

\theoremstyle{remark}

\numberwithin{equation}{section} 

\begin{document}
	\title{Abstract  hyperbolic chaos}

\author{Akhmet Marat}
\address{Middle East Technical University\\ Department of Mathematics\\ Ankara\\ Turkey}
\email{marat@metu.edu.tr}

\begin{abstract}    The  abstract  hyperbolic   sets    are  introduced.  Continuous   and differentiable mappings as well as  rate  of convergence  and transversal manifolds are not under  discussion,  and  the  symbolic dynamics paradigm  is realized in a new  way.  Our  suggestions   are  for  more  neat  comprehension  of chaos  in the   domain.  The  novelties   can  serve for  revisited models as well  as motivate new ones.
\end{abstract}
\subjclass[2010]{65P20;34H10;60G05;60G10}
\keywords{Abstract  hyperbolic  sets,  Bi-infinite indexing, Chaotic stable set,    Chaotic unstable set,  Poincar\'{e} chaos,  Li-Yorke chaos,  Devaney chaos,   Domain-structured chaos, Smale Horseshoe}
\maketitle

\section{Introduction}

Hyperbolic  sets  were  introduced by  S. Smale in famous papers \cite{Smale1, Smale2}. They  are   fundamental  for chaos  theory,   since  describe profoundly  domains,  where  the dynamics is happen.   The  genius  of H. Poincar\'{e},   who    explored  the  complexity  in celestial mechanics \cite{Poin},  and scrupulous  investigation  of  the Van der Pol  oscillator  by  M. Cartright, J. Littlewood and H. Levinson  \cite{Cartwright,Levinson}  caused  the  insights  of S. Smail.  Infinite  number of  hyperbolic  points in a bounded  area,  symbolic  dynamics,  manifolds, Markov partitions  and  differentiable mappings    make  the  background for  the theory \cite{Adler} - \cite{Williams}

Recently,  we have  introduced  a concept  of domain structured chaos  \cite{AkhmetDiscreteRandom, AkhmetDomainStruct},  which  may  become a useful  tool  for analysis  of complex dynamics \cite{AkhmetSimilarity}.   This  time,  the   new  concept  of \textit{abstract  hyperbolic  sets } is    considered to  continue  the  line of the research.   What   is  suggested,  is not   application of the symbolic  dynamics   through topological  equivalence,   which   requests  other   advanced conditions for  the  relation establishment  \cite{Devaney, Holmes, Newhouse}.  Moreover,   we do  not  utilize finite  partitioning    \cite{Bowen} and even countable  \cite {Bunimovich}  partitioning    for   involvement    of the symbolic  dynamics to  explore  the  complexity. We   perform   the  complete indexing  of the  domain.  In other  words,  infinitesimal  uncountable    partitioning,    and   determine  the  map,  which  orbits  are  constructed      through    shifting  in  indexes  such  that   the  research    can  be  focused on chaotic behavior.  That  is,   a  new   realization  of the  Bernoulli  map  is present.   By  utilizing   the\textit{ separation} and \textit{diameter  conditions}  \cite{AkhmetSimilarity, AkhmetDomainStruct},    we   explore   the  \textit{domain   structured  chaos }in the abstract  hyperbolic   sets.  The  discussion   initiated in this paper can    precise  chaotic dynamics description  as well  as  hyperbolic sets,     explore  relations,  which  either   have not  been  seen in former studies  or   can provide more sense to them.   For   example, in the present  paper,    a new proof  of   the Smale horseshoe chaos  is provided,    which  is free of  the  equivalence  application.  The  results    are    complemented  by   Poincar\'{e}  type  of chaos  and   isolated  dynamics on unstable  sets. 

	\section{Nonchaotic dynamics } 
\label{sec1}
In this  section we introduce    a formal  set  and a map, which   generalize what  is is known  for symbolic dynamics \cite{Wiggins88}.   Another suggestion  is  about    abstraction of the   concept  of hyperbolic sets \cite{Smale2,Bowen}, which    is understood, in our research,   as renouncing   some of  former properties and considering   new ones,  which  are  obvious,   and useful.

    Let the   space $(\mathcal F,  d)$  be given  with   metric $ d.$ The elements of the  set  are  labeled such  that      
    \begin{equation} \label{AbstFracSet1}
	\mathcal{F} =  \big\{\mathcal{F}_{\ldots i_{-2}i_{-1}i_{0}\centerdot i_1 i_2 \ldots}: i_k=1,2, ..., m, \; k= \ldots -2,-1,0,1, 2, ... \big\},
\end{equation}
where  $m$  is a  fixed natural number  and the  symbol $\centerdot$  is used to  determine  the  dynamics of the  present  research. 

The set  $\mathcal{F}$  is   called an   \textit{abstract  hyperbolic  set},      and  its elements    are  called \textit{hyperbolic  points}. The  set  of points,  
	\begin{equation} \label{AbstFracSet2}
		\mathcal{F}_{i_1 i_2 ...} =  \big\{\mathcal{F}_{\ldots j_{-2}j_{-1}j_{0}\centerdot i_1 i_2 ... i_n ... } : j_k=1,2, ..., m, \; k= \ldots -2,-1,0 \big\}
\end{equation}
is  the  \textit{stable set} for  a  fixed point   $\mathcal{F}_{\ldots i_{-k}\ldots i_{-2}i_{-1}i_{0}\centerdot i_1 i_2 ... i_n ... },$    and   the    collection   
  \begin{equation} \label{AbstFracSet3}
  	\mathcal{F}_{\ldots i_{-k}\ldots i_{-2}i_{-1}i_{0}} =  \big\{\mathcal{F}_{\ldots i_{-2}i_{-1}i_{0}\centerdot j_1 j_2 ... } : j_k=1,2, ..., m, \; k= 1,2,\ldots \big\}
  \end{equation}
  is  said to  be the  \textit{unstable   set}   of  the  point.

 The following  unions   are  needed,   in what  follows,  
\begin{equation} \label{AbstFracSubSet1}
\mathcal{F}_{i_{-k}i_{-k+1}\ldots i_ 0\centerdot i_1 i_2 ... i_n} = \bigcup_{j_k=1,2, ..., m } \mathcal{F}_{\ldots i_{-k-1}i_{-k}i_{-k+1}\ldots i_ 0\centerdot i_1 i_2 ... i_ni_{n+1}\ldots },
\end{equation}
where   indices  $ i_{-k},i_{-k+1},\ldots, i_ 0,i_1, i_2, ..., i_n,$    are  fixed, 
\begin{equation} \label{AbstFracSubSet2}
\mathcal{F}_{i_1 i_2 ... i_n} = \bigcup_{j_k=1,2, ..., m } \mathcal{F}_{\ldots i_{-k-1}i_{-k}i_{-k+1}\ldots i_ 0\centerdot i_1 i_2 ... i_ni_{n+1}\ldots },
\end{equation}
where indices  $i_1, i_2, ..., i_n,$  are fixed, and 
\begin{equation} \label{AbstFracSubSet3}
\mathcal{F}_{i_{-k}i_{-k+1}\ldots i_ 0} = \bigcup_{j_k=1,2, ..., m } \mathcal{F}_{\ldots i_{-k-1}i_{-k}i_{-k+1}\ldots i_ 0\centerdot i_1 i_2 ... i_ni_{n+1}\ldots },
\end{equation}
where indices  $ i_{-k},i_{-k+1},\ldots, i_ 0,$  are fixed.

We will  say   that  the \textit{diameter condition} is valid for the sets  if   for   every  fixed sequence $\ldots i_{-k}i_{-k+1}\ldots i_ 0\centerdot i_1 i_2 ... i_n\ldots$  it  is  true that 
\begin{equation} \label{Diamprop1}
\mathrm{diam}(\mathcal{F}_{i_{-k}i_{-k+1}\ldots i_ 0\centerdot i_1 i_2 ... i_n}) \to 0 \;\; \text{as} \;\; k, n \to \infty,
\end{equation}

where $ \mathrm{diam}(A) = \sup \{ d(\textbf{x}, \textbf{y}) : \textbf{x}, \textbf{y} \in A \} $, for a subset $ A $   of  $ \mathcal{F} $.

  Determine the  following   one-to-one and onto  map  $\varphi : \mathcal F \to  \mathcal F,$  such  that 
$$\varphi( \mathcal{F}_{\ldots i_{-k}\ldots i_{-2}i_{-1}i_{0}\centerdot i_1 i_2 ... i_n ... } ) =  \mathcal{F}_{\ldots i_{-k}\ldots i_{-2}i_{-1}i_{0} i_1  \centerdot i_2 ... i_n ... }.$$

Let  us fix  a point   $\mathcal{F}_{\ldots i_{-k}\ldots i_{-2}i_{-1}i_{0}\centerdot i_1 i_2 ... i_n ... }$  of the    hyperbolic set  and  an  arbitrary  element ,   $\mathcal{F}_{\ldots j_{-2}j_{-1}j_{0}\centerdot i_1 i_2 ... i_n ... },  $   from  the  stable manifold  of the point.   The  definition  of the  stable  set  and the  diameter  property    it follows   that  
$$d(\varphi^n(\mathcal{F}_{\ldots j_{-2}j_{-1}j_{0}\centerdot i_1 i_2 ... i_n ... }),\varphi^n (\mathcal{F}_{\ldots i_{-k}\ldots i_{-2}i_{-1}i_{0}\centerdot i_1 i_2 ... i_n ... } )) \to  0$$ as $n \to \infty.$
 
Similarly,   it  is true that  
$$d(\varphi^n(\mathcal{F}_{\ldots i_{-2}i_{-1}i_{0}\centerdot j_1 j_2 ... j_n ... }),\varphi^n (\mathcal{F}_{\ldots i_{-k}\ldots i_{-2}i_{-1}i_{0}\centerdot i_1 i_2 ... i_n ... } )) \to  0$$ as $n \to - \infty,$
for   points  of   the  unstable  manifold. 

It is clear that
\[ \mathcal{F} \supseteq \mathcal{F}_{i_1} \supseteq \mathcal{F}_{i_1 i_2} \supseteq ... \supseteq \mathcal{F}_{i_1 i_2 ... i_n} \supseteq \mathcal{F}_{i_1 i_2 ... i_n i_{n+1}} ... \]
and 
\[ \mathcal{F} \supseteq \mathcal{F}_{i_0} \supseteq \mathcal{F}_{i_{-1} i_0} \supseteq ... \supseteq \mathcal{F}_{i_{-n} i_{-n+1}... i_0} \supseteq \mathcal{F}_{i_{-n-1}i_{-n}... i_0} ..\, .\]

Considering iterations of the map, one can verify that
\begin{equation} \label{MapSubset1}
\varphi^n(\mathcal F_{i_1 i_2 ... i_n}) = \mathcal{F},
\end{equation}
and
\begin{equation} \label{MapSubset2}
\varphi^{-n}(\mathcal F_{i_{-n} i_{-n+1} ... i_0}) = \mathcal{F},
\end{equation}
for arbitrary natural number $ n .$  The relations (\ref{MapSubset1}) and (\ref{MapSubset2}) give us a reason to call $ \varphi $ a \textit{similarity map} and the number $ n $ the \textit{order of similarity}. We   called  it    \textit{abstract  similarity  map }  in  \cite{AkhmetSimilarity}.   

\section{Chaos}

Denote  by $ d(A, B)= \inf \{ d(\textbf{x}, \textbf{y}) : \textbf{x}\in A, \, \textbf{y} \in B \} $ the  function of    two bounded subsets  $ A $ and $ B $   of $ \mathcal{F} .$    The  space  $ (\mathcal{F},d) $ satisfies the \textit{separation condition }of degree $ n $ if there exist a positive number $ \varepsilon_0 $ and a natural number $ n $  such that for arbitrary   indices  $ i_1 i_2 ... i_n $ one can find   indices $ j_1 j_2 ... j_n $ such  that
\begin{equation} \label{C2}
d \big( \mathcal{F}_{i_1 i_2 ... i_n} \, , \, \mathcal{F}_{j_1 j_2 ... j_n} \big) \geq \varepsilon_0.
\end{equation}
A point $ \mathcal{F}_{\ldots i_{-k}\ldots i_{-2}i_{-1}i_{0}\centerdot i_1 i_2 ... i_n ... }$    from $\mathcal{F} $ is periodic with period $ n $ if its index consists of endless repetitions of a block of $ n $ terms.

\begin{theorem} \label{Thm1} If   the diameter  and  separation conditions are valid, then    the  similarity  map  $\varphi$  is chaotic on each  unstable set  of the dynamics in the sense of Devaney .
\end{theorem}

\begin{proof}  Utilizing the diameter condition, the transitivity will be proved if   there  exists   an element $\mathcal{F}_{\ldots i_{-k}\ldots i_{-2}i_{-1}i_{0}\centerdot i_1 i_2 ... i_n ... } $    of the  unstable  set  such that for any finite  sequence  $j_{-k}\ldots j_{-2}j_{-1}j_{0} j_1 j_2 ... j_n$   one can  find    that  for  a sufficiently   large   $ p,$     the  point  
	 $ \varphi^p(\mathcal{F}_{\ldots i_{-k}\ldots i_{-2}i_{-1}i_0\centerdot i_1 i_2 \ldots i_n \ldots })$     belongs to  $\mathcal{F}_{j_{-k}\ldots j_{-2}j_{-1}j_0\centerdot  j_1 j_2 ... j_n}.$    This  is true,   since   one can construct a sequence $ \ldots i_1 i_2 \ldots i_n \ldots,$    which   contains all the sequences of the type $ j_{-k}\ldots j_{-2}j_{-1}j_0 j_1 j_2 \ldots j_n$ as blocks.

	Fix a member $ \mathcal{F}_{\ldots i_{-k}\ldots i_{-2}i_{-1}i_{0}\centerdot i_1 i_2 ... i_n ... }$  of   the  unstable  set  and a positive number $ \delta $. Find a   sufficiently  large  number $ k $ such that $ \mathrm{diam}(\mathcal{F}_{i_1 i_2 ... i_k\centerdot i_1 i_2 ... i_k}) < \delta$ and choose a $ k $-periodic element $ \mathcal{F}_{\ldots i_1 i_2 ... i_k\centerdot  i_1 i_2 ... i_k ...} $ of $ \mathcal{F}_{i_1 i_2 ... i_k} $. It is clear that the periodic point is an $ \delta $-approximation for the considered member.   Thus,  the density of periodic points is proved.

	For sensitivity, fix a point $\mathcal{F}_{\ldots i_{-k}\ldots i_{-2}i_{-1}i_{0}\centerdot i_1 i_2 ... i_n ... }$  of the  unstable set   and an arbitrary positive number $ \varepsilon $.   Due to  the  diameter condition there exist an integer $ k $ and point  $\mathcal{F}_{\ldots i_{-k}\ldots i_{-2}i_{-1}i_{0}\centerdot i_1 i_2 ... i_k j_{k+1} j_{k+2} ...} \neq \mathcal{F}_{\ldots i_{-k}\ldots i_{-2}i_{-1}i_{0}\centerdot i_1 i_2 ... i_n ... } $ such that  $ d(\mathcal{F}_{\ldots i_{-k}\ldots i_{-2}i_{-1}i_{0}\centerdot i_1 i_2 ... i_n ... }, \mathcal{F}_{\ldots i_{-k}\ldots i_{-2}i_{-1}i_{0}\centerdot i_1 i_2 ... i_k j_{k+1} j_{k+2} ...}) < \varepsilon. $   We choose  $ j_{k+1}, j_{k+2}, ... $ such that $ d(\mathcal{F}_{i_{k+1} i_{k+2} ... i_{k+n}}, \mathcal{F}_{j_{k+1} j_{k+2} ... j_{k+n}}) > \varepsilon_0 $, by the separation condition. This proves the sensitivity.	
\end{proof}

In \cite{AkhmetPoincare}, Poisson stable motion is utilized to distinguish  chaotic behavior from  periodic motions in Devaney and Li-Yorke types.  The dynamics  is given the named  Poincar\'{e} chaos.  The next theorem shows that the Poincar\'{e} chaos is valid for the similarity dynamics.

\begin{theorem} \label{Thm2}    The   abstract  similarity  map  $\varphi$   is chaotic  in Poincar\'{e} sense  on each unstable  set  of the dynamics,  if    the diameter   and separation, properties   are  valid.
\end{theorem}  

The proof of the last theorem is based on the verification of Lemma 3.1 in \cite{AkhmetPoincare},  applied  to the similarity map.

In addition to the Devaney and Poincar\'{e}  chaos, it can be shown that the Li-Yorke chaos exists  fo   map $ \varphi $  on every  unstable  set  of the  dynamics. The proof  is similar to that of Theorem 6.35 in \cite{Chen}.   Thus, we  have obtained that    there  are Poncar\'{e},  Li-Yorke and Devaney  chaos on each unstable set  of the similarity  map.      Consequently,  the hyperbolic  set  $\mathcal F$   admits  the domain structured chaos.   It  is important  to  say  that  the    dynamics is  the  union  of  chaotic   unstable  sets.  Moreover,  the  chaos  can  be approved by  considering bi-infinite  symbolic dynamics,  but  we decide to utilize the  above  set-by-set way, since it provides more information on the  structure  of  the chaos.      It  is clear, also,  that   accepting  the separation property  for  decreasing  time,  one can  approve domain structured  chaos  for  stable  sets.

\section{Smale horseshoes}

As   an example,  consider the   Smale horseshoe map  \cite{Smale1}. We    will  discuss the    dynamics  omitting  many  details, which  can  be  found in \cite{Devaney}.   The Smale horseshoe map is the set of basic topological operations of stretching and folding.   The  result  of   the  iterations are   invariant  sets  
$$\Lambda_+ = \{q\in S:   F^k(q) \in S,  k =1,2,\ldots\}$$
and 
 $$\Lambda_- = \{q\in S:   F^{-k}(q) \in S,  k =1,2,\ldots\},$$ 
 where  $S$  is  the  unit  square  and $F$ is the  Smale map. 
  Sets    $\Lambda_+$  and   $\Lambda_-$  are  products of Cantor sets  with  a vertical  and a horizontal  lines respectively.  The  domain  under discussion   is  $\Lambda = \Lambda_+ \cap \Lambda_-.$      In literature,  the   dynamics ${(\Lambda, F)}$  is  proved to  be    topologically  equivalent  to  the dynamics of $(\Sigma_2, \sigma),$  where  $\Sigma_2$   is the  set  of two-sided sequences of $0'$s and $1'$s  and  $\sigma$  is the  Bernoulli shift.   Considering  the  equivalence,   chaos   was proved  for  the  map  $F.$    In our  case,   to  accomplish  the  analysis   we utilize the experience just    for the indexing  points of the  set $\Lambda,$  and  constructing   the  abstract  hyperbolic  set  $\mathcal F$    with the  Euclidean planar  distance.  The  construction   in  \cite{Devaney} demonstrates  that  the  diameter  and separation conditions are  valid for  the  abstract  hyperbolic set.      Consequently,    the  domain structured  chaos  of the  similarity  map  is proper for  the  set  $\Lambda.$    Apparently,    the way  of chaos  generation  can  be applied to  other  sets  with  hyperbolic  structure \cite{Adler}-\cite{Williams}.

  \section{Conclusion}  We   suggest  an  abstraction for two-sided dynamical  processes with discrete time.  A method  of chaos generation  is provided as  dynamics of the abstract  similarity  map.  Thus,  all the three types of Poincar\'{e},  Li-Yorke and Devaney  chaos  are  approved in the  domain.  The approach  is also of strong   practical sense. For instance, it  has been  realized in our papers \cite{AkhmetDomainStruct,AkhmetSimilarity} for fractals and neural networks.   This time, it is proven that the  hyperbolic  set  is  a  union of chaotic unstable  sets. Thus,   comprehension of the  dynamical   has been deepen.  It  is important  that    the  chaotic  behavior   can  be extend   for the   decreasing  time,  if  one assume  the   similar  separation  condition.  Another   benefit   is that   multidimensional  generalizations   are  possible on the  basis of proposals.    The  processes  in the focus  are  deterministic.   Not  only   the  deep  past,  but  the  very  far   future   has to  be determined for the  abstract  similarity  map  application.  That is,   the  hyperbolic dynamics  considers processes with infinite  chaotic history.     Nevertheless,  the approach  does not  exclude  possibility  of random dynamics   description  \cite{AkhmetDiscreteRandom}.

\

\end{document}